\documentclass[11 pt]{amsart}
\usepackage{amsfonts}
\usepackage{ifthen}
\usepackage{amsthm}
\usepackage{amsmath}
\usepackage{graphicx}
\usepackage{amscd,amssymb,amsthm}
\usepackage{graphicx}
\usepackage{epstopdf}
\usepackage{hyperref}
\usepackage{cleveref}
\usepackage{cite}

\newcounter{minutes}
\divide\time by 60
\newcounter{hours}
\multiply\time by 60 \addtocounter{minutes}{-\time}

\newtheorem{lemma}{Lemma}[section]
\newtheorem{theorem}{Theorem}[section]

\newtheorem{remark}{Remark}

\newcommand{\real}{\operatorname{Re}}

\newcommand{\loga}{\operatorname{Log}}


\begin{document}

\title{Radii of starlikeness and convexity of generalized $\emph{k}-$Bessel functions}

\author[E. Toklu]{Evr{\.I}m Toklu}
\address{Department of Mathematics, Faculty of Education, A\u{g}r{\i} {\.I}brah{\i}m \c{C}e\c{c}en University, 04100 A\u{g}r{\i}, Turkey} 
\email{evrimtoklu@gmail.com}

\def\thefootnote{}
\footnotetext{ \texttt{File:~\jobname .tex,
printed: \number\year-0\number\month-\number\day,
\thehours.\ifnum\theminutes<10{0}\fi\theminutes}
} \makeatletter\def\thefootnote{\@arabic\c@footnote}\makeatother

\maketitle

\begin{abstract}
The main purpose of this paper is to determine the radii of starlikeness and convexity of the generalized $\emph{k}-$Bessel functions for three different kinds of normalization by using their Hadamard factorization in such a way that the resulting functions are analytic in the unit disk of the complex plane. The characterization of entire functions from Laguerre-P\'olya class plays an crucial role in this paper. Moreover, the interlacing properties of the zeros of $\emph{k}-$Bessel function and its derivative is also useful in the proof of the main results. By making use of the Euler-Rayleigh inequalities for the real zeros of the generalized $\emph{k}-$Bessel function, we obtain some tight lower and upper bounds for the radii of starlikeness and convexity of order zero.

\end{abstract}

\section{\bf Introduction and The Main Results}
It is well known fact that special functions have an indispensable place in the solution of a wide variety of important problems. Due to the versatile properties of special functions, it is important to examine their properties in many aspects. In the recent years, there has been a vivid interest on geometric properties of some special functions from the point of view of geometric function theory. Baricz and his coauthors investigated in details the determination of the radii of starlikeness and convexity of some normalized forms of these special functions, see for example \cite{BTK}, \cite{BDOY}, \cite{BKS}, \cite{ABO}, \cite{ABY}, \cite{BOS}, \cite{BSz14}, \cite{BSz15}, \cite{BSz16}, \cite{BP} and the references therein for more details. If these studies are analysed in-depth, it can clearly be seen that  the radii of univalence, starlikeness and convexity are obtained as solutions of some transcendental equations and the obtained radii satisfy some interesting inequalities. In addition, the other main fact seen on these studies is that the positive zeros of Bessel, Struve, Lommel functions of the first kind and the Laguerre-P\'olya class of entire functions have a great impotance in these papers. It is important to mention that in recent years, there has been extensively interested on the $k-$calculus. Actually, the origins of what can be called the $k-$calculus are based on the definition introduced by Diaz and Pariguan ( see cf. \cite{DP07} ) of the $k-$gamma function and the Pochhammer $k-$symbol as generalizations of the well known functions the classical gamma function and the classical Pochhammer symbol. Since then there are many works devoted to studying generalizations of some of well known special functions. So can be found the $k-$Beta function, the $k-$Zeta function and the $k-$Wright function. Recently, Mondal and Akel in \cite{MA18} introduced and studied a generalization of the $k-$Bessel function of order $\nu.$ And also, they investigated monotonicity and log-convexity properties of the generalized $k-$Bessel function ${}_{k}W_{\nu,c}.$ 

Motivated by the above series of papers on geometric properties of special functions, in this paper our aim is to present some similar results for the normalized forms of the generalized $k-$Bessel functions. For this, three different normalizations are applied in such way that the resulting functions are analytic. By considering the Hadamard factorization of the generalized $k-$Bessel function  and combining the methods from \cite{BKS}, \cite{BDOY}, \cite{BSz14} and \cite{BTK}, we determine the radii of starlikeness and
convexity for each of the three functions.

Before starting to present our main results, we would like to draw attention to some basic concepts needed for building our main results. For $r>0$ we denote by $\mathbb{D}_r=\left\{z\in\mathbb{C}: |z|<r\right\}$ the open disk of radius $r$ centered at the origin. Let $f:\mathbb{D}_r\to\mathbb{C}$ be the function defined by
\begin{equation}
f(z)=z+\sum_{n\geq 2}a_{n}z^{n},  \label{eq0}
\end{equation}
where $r$ is less or equal than the radius of convergence of the above power series. Denote by $\mathcal{A}$ the class of allanalytic functions of the form \eqref{eq0}, that is, normalized by the conditions $f(0)=f^{\prime}(0)-1=0.$ We say that the function $f,$ defined by \eqref{eq0}, is starlike function in $\mathbb{D}_r$ if $f$ is univalent in $\mathbb{D}_r$, and the image domain $f(\mathbb{D}_r)$ is a starlike domain in $\mathbb{C}$ with respect to the origin (see \cite{Dur} for more details). Analytically, the function $f$ is starlike in $\mathbb{D}_r$ if and only if $$\real\left( \frac{zf^{\prime }(z)}{f(z)}\right) >0 \quad \mbox{for all}\ \ z\in\mathbb{D}_r.$$ For $\alpha \in [0,1)$ we say that the function $f$ is starlike of order $\alpha $ in $\mathbb{D}_r$ if and only if 
$$\real\left( \frac{zf^{\prime }(z)}{f(z)}\right) >\alpha \quad \mbox{for all}\ \ z\in\mathbb{D}_r.$$
The radius of starlikeness of order $\alpha$ of the function $f$ is defined as the real number
\begin{equation*}
r_{\alpha }^{\star}(f)=\sup\left\{r>0\left|\real\left(\frac{zf^{\prime }(z)}{f(z)}\right)  >\alpha \;\text{for all }z\in	\mathbb{D}_r\right.\right\}.
\end{equation*}
Note that $r^{\star}(f)=r_{0}^{\star}(f)$ is in fact the largest radius such that the image region $f(\mathbb{D}_{r^{\star}(f)})$ is a starlike domain with respect to the origin.
The function $f,$ defined by \eqref{eq0}, is convex in the disk $\mathbb{D}_r$ if $f$ is univalent in $\mathbb{D}_r$, and the image domain $f(\mathbb{D}_r)$ is a convex domain in $\mathbb{C}.$ Analytically, the function $f$ is convex in $\mathbb{D}_r$ if and only if
$$\real\left(  1+\frac{zf^{\prime \prime }(z)}{f^{\prime }(z)}\right)>0  \quad \mbox{for all}\ \ z\in\mathbb{D}_r.$$
For $\alpha \in[0,1)$ we say that the function $f$ is convex of order $\alpha $ in $\mathbb{D}_r$ if and only if 
$$\real\left( 1+\frac{zf^{\prime \prime }(z)}{f^{\prime }(z)}\right)
>\alpha \quad \mbox{for all}\ \ z\in\mathbb{D}_r.$$ 
We shall denote the radius of convexity of order $\alpha $ of the function $f$ by the real number
\begin{equation*}
r_{\alpha }^{c}(f)=\sup \left\{ r>0 \left|\real\left( 1+
\frac{zf^{\prime \prime }(z)}{f^{\prime }(z)}\right) >\alpha \;\text{for all }z\in\mathbb{D}_r\right.\right\} .
\end{equation*}
Note that $r^{c}(f)=r_{0}^{c}(f)$ is the largest radius such that the image region $f(\mathbb{D}_{r^{c}(f)})$ is a convex domain.

We recall that a real entire function $q$ belongs to the  Laguerre-P\'{o}lya class $\mathcal{LP}$ if it can be represented in the form $$q(x)=cx^{m}e^{-ax^2+bx}\prod_{n\geq1}\left(1+\frac{x}{x_n}\right)e^{-\frac{x}{x_n}},$$ with $c,b,x_n\in\mathbb{R}, a\geq0, m\in\mathbb{N}_0$ and $\sum\frac{1}{{x_n}^2}<\infty.$ We note that the class $\mathcal{LP}$ is the complement of the space of polynomials whose zeros are all real in the topology induced by the uniform convergence on the compact sets of the complex plane of polynomials with only real zeros. For more details on the class $\mathcal{LP}$ we refer to \cite[p. 703]{DC} and to the references therein.

\subsection{Generalized $k-$Bessel function} \label{Section2}
In this section we shall focus on a generalization of the $k-$Bessel function of order $\nu$ defined by the series
\begin{equation}\label{kBessel1}
{}_{k}W_{\nu,c}(z)=\sum_{n=0}^{\infty}\frac{(-c)^n}{n!\Gamma_{k}(nk+\nu+k)}\left( \frac{z}{2}\right)^{2n+\frac{\nu}{k}},
\end{equation}
where $k>0,$ $\nu>-1,$ $c\in\mathbb{R}$ and  $\Gamma_{k}$ stands for the $k-$gamma functions studied in \cite{DP07} and defined by
\[\Gamma_{k}(z)=\int_{0}^{\infty}t^{z-1}e^{-\frac{t^{k}}{k}} dt, \]
for $\real(z)>0.$ For several intriguing properties of  $k-$Bessel functions one can consult on \cite{MA18}. Moreover, several properties of the $k-$gamma functions in generalizing other related functions like $k-$beta and $k-$digamma functions can be found in \cite{DP07}, \cite{MNR13}, \cite{NP14} and references therein. It is important to mention that for a complex number $z$ and a positive real number $k,$ the $k-$gamma function and the classical Gamma function have the relation
\begin{equation}\label{RelationG-kG}
\Gamma_{k}(z)=k^{\frac{z}{k}-1}\Gamma\left( \frac{z}{k}\right). 
\end{equation}
It is important to note that for a positive real number $k,$ the $k-$gamma function satisfies the following properties
\begin{align}
\Gamma_{k}(z+k)&=z\Gamma_{k}(z), \text{ \ \ }\Gamma_{k}(k)=1 \label{RelationG-KG2} \text{ \ and \ }\\
\frac{1}{\Gamma_{k}(z)}&=zk^{-\frac{z}{k}}e^{\frac{z}{k}\gamma}\prod_{n\geq 1}\left(1+\frac{z}{nk}\right)e^{-\frac{z}{nk}},  \label{RelationG-KG3}
\end{align}
where $\gamma$ is Euler's constant.

Observe that as $k\rightarrow 1,$ the $k-$Bessel function ${}_{1}W_{\nu,1}$ is reduced to the classical Bessel function $J_{v}$, whereas ${}_{1}W_{\nu,-1}$ coincides with the modified Besel function $I_{\nu}.$

It is easy to check that the function $z\mapsto{}_{k}W_{\nu,c}$ does not belong to the class $\mathcal{A}.$ Thus first we shall perform some natural normalization. We define three functions originating from ${}_{k}W_{\nu,c}(.):$
\begin{align*}
{}_{k}f_{\nu,c}(z)=\left(2^{\frac{\nu}{k}}\Gamma_{k}(\nu+k){}_{k}W_{\nu,c}(z)\right)^{\frac{k}{\nu}},  \\
{}_{k}g_{\nu,c}(z)=2^{\frac{\nu}{k}}\Gamma_{k}(\nu+k)z^{1-\frac{\nu}{k}} {}_{k}W_{\nu,c}(z),\\
{}_{k}h_{\nu,c}(z)=2^{\frac{\nu}{k}}\Gamma_{k}(\nu+k)z^{1-\frac{\nu}{2k}}{}_{k}W_{\nu,c}(\sqrt{z}).
\end{align*}
It is obvious that each of these functions are of the class $\mathcal{A}.$ Of course, it can be written infinitely many other normalization; the main motivation to consider the above ones is the studied normalization in the literature of Bessel, Struve, Lommel and Wright functions. Moreover, it is convenient to mention here that in fact 
\[{}_{k}f_{\nu,c}(z)=\exp\left[\frac{k}{\nu}\loga(2^{\frac{\nu}{k}}\Gamma_{k}(\nu+k){}_{k}W_{\nu,c}(z))\right], \]
where $\loga$ represents the principle branch of the logarithm function and every many-valued function considered in this paper are taken with the principal branch.

The following lemma,  which we believe is of independent interest, plays a crucial role in proving our main results which are related to radii of starlikeness and convexity of functions ${}_{k}f_{\nu,c}$, ${}_{k}g_{\nu,c}$, and ${}_{k}h_{\nu,c}.$
\begin{lemma} \label{kBesselLemma}
Let $k>0$, $c>0$ and $\nu>0$. Then the function $z \mapsto {}_{k}W_{\nu,c}(z) $ has infinitely many zeros which are all real. Denoting by ${}_{k}\omega_{\nu,c,n}$ the $n$th positive zero of ${}_{k}W_{\nu,c}(z)$, under the same conditions the Weierstrassian decomposition
\begin{equation}\label{InfiniteProductkBessel}
{}_{k}W_{\nu,c}(z)=\frac{\left( \frac{z}{2}\right) ^{\frac{\nu}{k}}}{\Gamma_{k}(\nu+k)}\prod_{n\geq 1}\left(1-\frac{z^2}{{}_{k}\omega_{\nu,c,n}^2}\right) 
\end{equation}
is fulfilled, and this product is uniformly convergent on compact subsets of the complex plane. Moreover, if we denote by ${}_{k}\omega_{\nu,c,n}^{\prime}$ the nth positive zero of ${}_{k}W_{\nu,c}^{\prime}(z)$, then positive zeros of ${}_{k}W_{\nu,c}$ are interlaced with those of ${}_{k}W_{\nu,c}^{\prime}.$ In other words, the zeros satisfy the chain of inequalities	$${}_{k}\omega_{\nu,c,1}^{\prime}<{}_{k}\omega_{\nu,c,1}<{}_{k}\omega_{\nu ,c,2}^{\prime}<{}_{k}\omega_{\nu,c,2}<\dots_{.}$$
\end{lemma}
\begin{proof}
Let us start to prove by showing the reality of zeros of the generalized $k-$Bessel function ${}_{k}W_{\nu,c}(z).$ For fulfilling this objective, consider the entire function
\[{}_{k}W_{\nu,c}(z)=\left( \frac{z}{2}\right)^{\frac{\nu}{k}} \sum_{n=0}^{\infty}\frac{(-c)^n}{n!\Gamma_{k}(nk+\nu+k)}\left( \frac{z}{2}\right)^{2n}.\]
The function ${}_{k}G_{\nu,c}:\left[0,\infty \right) \rightarrow \mathbb{R} $ defined by
\[{}_{k}G_{\nu,c}(z)= \frac{1}{\Gamma_{k}(nk+\nu+k)}\]
is entire function and of growth order $1$ (see Eqn. \eqref{RelationG-KG3}), belongs to $\mathcal{LP}.$ Moreover, if we choose $f(z)=e^{-c\left(\frac{z}{2}\right)^{2} },$ which has no zeros at all, then with the aid of the Runckel's theorem stated in \cite[Lemma 4, p.p 2209]{BS18} we say that the  generalized $k-$Bessel function ${}_{k}W_{\nu,c}(z)$ has real zeros only if  $k>0$, $c>0$ and $\nu>0$. Furthermore, taking into account that
$$ c_n=\frac{(-c)^n}{n!\Gamma_{k}(nk+\nu+k)} \text{ \ and \ } \Gamma(n+1)=n!,$$ 
the growth order of the generalized $k-$Bessel function is calculated as
\begin{align*}
\rho({}_{k}W_{\nu,c})&=\limsup_{n\rightarrow \infty}\frac{n\log n}{-\log\left| c_n\right|}=\limsup_{n\rightarrow \infty}\frac{n\log n}{-\log\left|\frac{(-c)^n}{n!\Gamma_{k}(nk+\nu+k)} \right|}  \\
&=\limsup_{n\rightarrow \infty} \frac{n\log n}{-n\log c+\log \Gamma(n+1)+\Gamma_{k}(nk+\nu+k)} \quad (\text{ \ by Eqn. \eqref{RelationG-kG}})\\
&=\frac{1}{2}.
\end{align*}
It is well known that the finite growth order $\rho$ of an entire function is not equal to a positive integer, then  the function has infinitely many zeros. That is to say, $k-$Bessel function ${}_{k}W_{\nu,c}(z)$ given in \eqref{kBessel1} has infinitely zeros. As a result of these explanations, we deduce that the zeros of the $k-$Bessel function ${}_{k}W_{\nu,c}(z)$ are all real. In this case, by means of the Hadamard theorem on growth order of the entire function, it follows that its infinite product representation is exactly what we have in \Cref{kBesselLemma}. Fnally, because of the fact that the growth order $\rho({}_{k}W_{\nu,c})=\frac{1}{2}$ is not an integer, we conclude that the {\it genus} of the $k-$Bessel function ${}_{k}W_{\nu,c}(z)$ is equal to zero which is the integer part of $\rho({}_{k}W_{\nu,c}).$ Then the zeros of ${}_{k}W_{\nu,c}'(z)$ are all real also and are separated from each other by those of ${}_{k}W_{\nu,c}(z).$  More precisely, taking into account the infinite product representation we get
\begin{equation}\label{Realzeros1}
\frac{{}_{k}W_{\nu,c,}'(z)}{_{k}W_{\nu,c}(z)}=\frac{\nu}{kz}-\sum_{n \geq 1}\frac{2z}{{}_{k}\omega_{\nu,c,n}^2 -z^2}.
\end{equation}
Diferentiating both sides of \eqref{Realzeros1} we arrive at
$$\frac{d}{dz}\left(\frac{_{k}W_{\nu,c}'(z)}{_{k}W_{\nu,c}(z)}\right)=-\frac{\nu}{kz^2}-2\sum_{n\geq1}\frac{{}_{k}\omega_{\nu,c,n}^2+z^2}{\left({}_{k}\omega_{\nu,c,n}^2-z^2\right) ^2}.$$
Since the expression on the right-hand side is real and negative for $z$ real, the quotient $\frac{_{k}W_{\nu,c}'(z)}{_{k}W_{\nu,c}(z)}$ is a strictly decreasing function from $+\infty$ to $-\infty$ as $z$ increases through real values over the open interval $({}_{k}\omega_{\nu,c,n}, {}_{k}\omega_{\nu,c,n+1}) \text{ \ \ } n\in \mathbb{N}.$ That is to say that the function $_{k}W_{\nu,c}'(z)$ vanishes just once between two consecutive zeros of the function $_{k}W_{\nu,c}(z).$ In other words, the zeros satisfy the chain of inequalities	$${}_{k}\omega_{\nu,c,1}^{\prime}<{}_{k}\omega_{\nu,c,1}<{}_{k}\omega_{\nu ,c,2}^{\prime}<{}_{k}\omega_{\nu,c,2}<\dots,$$
where ${}_{k}\omega_{\nu,c,n}$ and ${}_{k}\omega_{\nu,c,n}^{\prime}$ are, respectively, the $n$th positive zeros of ${}_{k}W_{\nu,c}(z)$ and ${}_{k}W_{\nu,c}'(z).$

The proof of the Lemma is completed.
\end{proof}

\subsection{The radii of  starlikeness of order $\alpha$ of functions ${}_{k}f_{\nu,c}$, ${}_{k}g_{\nu,c}$, and ${}_{k}h_{\nu,c}$} This section is devoted to investigate the radii of starlikeness of order $\alpha$ of the normalized forms of the  $k-$Bessel functions  ${}_{k}W_{\nu,c}(z),$ that is of ${}_{k}f_{\nu,c}$, ${}_{k}g_{\nu,c}$, and ${}_{k}h_{\nu,c}$. In addition, in this section we aim to find some tight lower and upper bounds for the radii of starlikeness and convexity of order zero.
\begin{theorem}\label{MainTheorem1}
Let $k>0$, $c>0,$ $\nu>0$ and $\alpha \in \left[0,1 \right) $. Then the following assertions are true.
\begin{itemize}
	\item [\bf a.] The radius of starlikeness of order $\alpha$ of the function ${}_{k}f_{\nu,c}$ is the smallest positive root of the equation
	\[kr{}_{k}W_{\nu,c}'(r)-\nu \alpha {}_{k}W_{\nu,c}(r)=0.\]
	\item [\bf b.] The radius of starlikeness of order $\alpha$ of the function ${}_{k}g_{\nu,c}$ is the smallest positive root of the equation
	\[ r{}_{k}W_{\nu,c}'(r)-\left( \alpha+\frac{\nu}{k}-1\right) {}_{k}W_{\nu,c}(r)=0 .\]
	\item [\bf c.] The radius of starlikeness of order $\alpha$ of the function ${}_{k}h_{\nu,c}$ is the smallest positive root of the equation
	\[\sqrt{r}{}_{k}W_{\nu,c}'(\sqrt{r})-2\left(\alpha+\frac{\nu}{2k}-1\right) {}_{k}W_{\nu,c}(\sqrt{r})=0 .\]
\end{itemize}
\end{theorem}
\begin{proof}
In order to verify assertions of the theorem we need to show that the inequalities
\begin{equation}\label{Section1Eq.1}
\real\left(\frac{z{}_{k}f_{\nu,c}'(z)}{_{k}f_{\nu,c}(z)}\right)>\alpha, \quad \real\left(\frac{z{}_{k}g_{\nu,c}'(z)}{_{k}g_{\nu,c}(z)}\right)>\alpha \text{ \ and \ } \real\left( \frac{z{}_{k}h_{\nu,c}'(z)}{_{k}h_{\nu,c}(z)}\right)>\alpha,
\end{equation}
are valid for $z\in \mathbb{D}_{{}_{k}x_{\nu,c,1}}(_{k}f_{\nu,c}),$  $z\in \mathbb{D}_{{}_{k}y_{\nu,c,1}}(_{k}g_{\nu,c})$ and  $z\in \mathbb{D}_{{}_{k}z_{\nu,c,1}}(_{k}h_{\nu,c}),$ respectively, and each of the above-mentioned inequalities does not hold in any larger disk. It is important to  that under the corresponding conditions the zeros of the $k-$Besel function ${}_{k}W_{\nu,c}(z)$ are all real. As a result of this reminding and in light of the \Cref{kBesselLemma}, the $k-$Bessel function can be represented by the Weierstrassian decomposition of the form
\[ {}_{k}W_{\nu,c}(z)=\frac{\left( \frac{z}{2}\right) ^{\frac{\nu}{k}}}{\Gamma_{k}(\nu+k)}\prod_{n\geq 1}\left(1-\frac{z^2}{{}_{k}\omega_{\nu,c,n}^2}\right) \]
and this infinite product is uniformly convergent on each compact subset of $\mathbb{C}.$ Now, consider the functions
\begin{align*}
{}_{k}f_{\nu,c}(z)=\left(2^{\frac{\nu}{k}}\Gamma_{k}(\nu+k){}_{k}W_{\nu,c}(z)\right)^{\frac{k}{\nu}},  \\
{}_{k}g_{\nu,c}(z)=2^{\frac{\nu}{k}}\Gamma_{k}(\nu+k)z^{1-\frac{\nu}{k}} {}_{k}W_{\nu,c}(z),\\
{}_{k}h_{\nu,c}(z)=2^{\frac{\nu}{k}}\Gamma_{k}(\nu+k)z^{1-\frac{\nu}{2k}}{}_{k}W_{\nu,c}(\sqrt{z}).
\end{align*}
Logarithmic differentiation of both sides of each of the above functions implies in turn
\begin{align*}
\frac{z{}_{k}f_{\nu,c}'(z)}{{}_{k}f_{\nu,c}(z)}=& \frac{k}{\nu}\left(\frac{z{}_{k}W_{\nu,c}'(z)}{{}_{k}W_{\nu,c}(z)} \right)=1-\frac{k}{\nu}\sum_{n \geq 1}\frac{2z^2}{{}_{k}\omega_{\nu,c,n}^2-z^2},  \\
\frac{z{}_{k}g_{\nu,c}'(z)}{{}_{k}g_{\nu,c}(z)}=& 1-\frac{\nu}{k}+\left(\frac{z{}_{k}W_{\nu,c}'(z)}{{}_{k}W_{\nu,c}(z)} \right)=1-\sum_{n\geq1}\frac{2z^2}{{}_{k}\omega_{\nu,c,n}^2-z^2}, \\
\frac{z{}_{k}h_{\nu,c}'(z)}{{}_{k}h_{\nu,c}(z)}=&1-\frac{\nu}{2k}+\frac{1}{2}\left(\frac{\sqrt{z}{}_{k}W_{\nu,c}'(\sqrt{z})}{{}_{k}W_{\nu,c}(\sqrt{z})} \right)=1-\sum_{n \geq 1}\frac{z}{{}_{k}\omega_{\nu,c,n}^2-z}.
\end{align*}
From \cite{BKS} we know that if $z\in \mathbb{C}$ and $\theta \in \mathbb{R}$ are such that $\theta >\left| z\right|,$ then 
\begin{equation}\label{Section1Eq.2}
\frac{\left| z\right|}{\theta-\left| z\right|}\geq \real\left( \frac{z}{\theta-z}\right).
\end{equation}
By using Eqn. \eqref{Section1Eq.2}, for  $k>0,$ $c>0,$ $\nu>0,$ $n\in\mathbb{N}$ and $\left| z\right|<{}_{k}\omega_{\nu,c,1}$ we get that
\begin{align*}
\real\left(\frac{z{}_{k}f_{\nu,c}'(z)}{_{k}f_{\nu,c}(z)}\right)&=1-\frac{k}{\nu}\real\left(\sum_{n \geq 1}\frac{2z^2}{{}_{k}\omega_{\nu,c,n}^2-z^2} \right) \\
&\geq 1-\frac{k}{\nu}\sum_{n \geq 1}\frac{2\left| z\right| ^2}{{}_{k}\omega_{\nu,c,n}^2-\left| z\right| ^2} = \frac{\left| z\right| {}_{k}f_{\nu,c}'(\left| z\right| )}{_{k}f_{\nu,c}(\left| z\right| )},\\
\real\left(\frac{z{}_{k}g_{\nu,c}'(z)}{{}_{k}g_{\nu,c}(z)}\right)&=1-\real\left(\sum_{n \geq 1}\frac{2z^2}{{}_{k}\omega_{\nu,c,n}^2-z^2} \right)\\
&\geq 1- \sum_{n \geq 1}\frac{2\left| z\right| ^2}{{}_{k}\omega_{\nu,c,n}^2-\left| z\right| ^2}= \frac{\left| z\right| {}_{k}g_{\nu,c}'(\left| z\right| )}{{}_{k}g_{\nu,c}(\left| z\right|)},\\
\real\left(\frac{z{}_{k}h_{\nu,c}'(z)}{_{k}h_{\nu,c}(z)}\right)&=1-\real\left(\sum_{n\geq1}\frac{z}{{}_{k}\omega_{\nu,c,n}^2-z}\right)\\
&\geq1-\sum_{n\geq1}\frac{\left|z\right|}{{}_{k}\omega_{\nu,c,n}^2-\left| z\right|}=\frac{\left|z\right|{}_{k}h_{\nu,c}'(\left| z\right| )}{{}_{k}h_{\nu,c}(\left| z\right| )}.
\end{align*}
It is important to mention that equalities in the above-mentioned inequalities are attained only when $z=\left| z\right|=r.$ In light of the previous inequalities and the minimum principle for harmonic functions we deduce that the inequalities stated in \eqref{Section1Eq.1} hold if and only if $\left|z\right|<{}_{k}x_{\nu,c,1}, $ $\left|z\right|<{}_{k}y_{\nu,c,1} $ and $\left|z\right|<{}_{k}z_{\nu,c,1} $, respectively, where ${}_{k}x_{\nu,c,1},$ ${}_{k}y_{\nu,c,1}$ and ${}_{k}z_{\nu,c,1}$ are the smallest positive roots of the equations
\begin{equation} \label{Section1Eq.3}
\frac{r{}_{q}f_{p,b,c,\delta}'(r)}{_{q}f_{p,b,c,\delta}(r)}=\alpha, \text{ \ \ } \frac{r{}_{q}g_{p,b,c,\delta}'(r)}{_{q}g_{p,b,c,\delta}(r)}=\alpha \text{ \ and \ } \frac{r{}_{q}h_{p,b,c,\delta}'(r)}{_{q}h_{p,b,c,\delta}(r)}=\alpha.
\end{equation}
which are equivalent to
\[kr{}_{k}W_{\nu,c}'(r)-\nu \alpha {}_{k}W_{\nu,c}(r)=0, \text{ \ \ } r{}_{k}W_{\nu,c}'(r)-\left( \alpha+\frac{\nu}{k}-1\right) {}_{k}W_{\nu,c}(r)=0 \]
and
\[\sqrt{r}{}_{k}W_{\nu,c}'(\sqrt{r})-2\left( \alpha+\frac{\nu}{2k}-1\right) {}_{k}W_{\nu,c}(\sqrt{r})=0.\]

This completes the proof of the theorem.
\end{proof}
\begin{remark}
It is evident that our main results which are given in \Cref{MainTheorem1}, in particular for $c=1$ and $k=1,$ correspond to the results in \cite[Theorem 1]{BKS}.
\end{remark}
The following theorem provides some tight lower and upper bounds for the radii
of starlikeness of the functions considered in the above theorem. The technique used in determining the bounds for the radii of starlikeness of these functions is based on the study of \cite{ABO} and \cite{ABY}. The main idea in the proof of the next theorem is to determine some Euler-Rayleigh inequalities for the first positive zero of some entire functions, which are connected with the transcendental equations appeared in the above theorem. Of course, it is possible to give more tighter bounds in the next theorem by using higher order Euler-Rayleigh inequalities for $k\in\left\lbrace2,3,... \right\rbrace $, however we omitted them owing to their complicated form.
\begin{theorem} \label{MainTheorem2}
Let $k>0$, $c>0$ and $\nu>0$.
\begin{itemize}
	\item [\bf a.] The radius of starlikeness $r^{\star}({}_{k}f_{\nu,c})$ is satisfies the inequalities
	\[2\sqrt{\frac{\nu(\nu+k)}{c(\nu+2k)}}<r^{\star}({}_{k}f_{\nu,c})<2(\nu+2k)\sqrt{\frac{k(\nu+k)}{c\bigg((\nu+2k)^3-k\nu(\nu+k)(\nu+4k)\bigg)}}.\]
	
	\item [\bf b.] The radius of starlikeness $r^{\star}({}_{k}g_{\nu,c})$ is satisfies the inequalities
	\[2\sqrt{\frac{\nu+3}{3c}}<r^{\star}({}_{k}g_{\nu,c})<2\sqrt{\frac{3(\nu+k)(\nu+2k)}{c\left( 4\nu+13k\right) }}.\] 
	
	\item [\bf c.] The radius of starlikeness $r^{\star}({}_{k}h_{\nu,c})$ is satisfies the inequalities
	\[\frac{2(\nu+k)}{c}<r^{\star}({}_{k}h_{\nu,c})<\frac{8(\nu+k)(\nu+2k)}{c(\nu+5k)}.\]
	
\end{itemize}
\end{theorem}
\begin{proof}
\begin{itemize}
	\item [\bf a.] If we take $\alpha=0$ in \Cref{MainTheorem1}, then the radius of starlikeness of the normalized $k-$Bessel function ${}_{k}f_{\nu,c}$ corresponds to the radius of starlikeness of the function ${}_{k}\Xi_{\nu,c}(z)={}_{k}W_{\nu,c}'(z).$ The infinite series representation of the function ${}_{k}\Xi_{\nu,c}(z)$ and its derivative are as follows:
	\begin{equation}\label{BounsforStar1}
	{}_{k}\Xi_{\nu,c}(z)=\sum_{n\geq 0}\frac{(-c)^n (2n+\frac{\nu}{k})}{2^{2n+\frac{\nu}{k}}n!\Gamma_{k}(nk+\nu+k)}z^{2n+\frac{\nu}{k}-1}
	\end{equation}
	and
	\begin{equation}\label{BounsforStar2}
	{}_{k}\Xi_{\nu,c}'(z)=\sum_{n\geq 0}\frac{(-c)^n (2n+\frac{\nu}{k})(2n+\frac{\nu}{k}-1)}{2^{2n+\frac{\nu}{k}}n!\Gamma_{k}(nk+\nu+k)}z^{2n+\frac{\nu}{k}-2}.
	\end{equation}
	In light of \Cref{kBesselLemma} we deduce that the function $z\mapsto 	{}_{k}W_{\nu,c}(z)$ is of the Laguerre-P\'olya class $\mathcal{LP}.$ It is well known that this class of entire functions is closed under differentiation, and therefore  $z\mapsto {}_{k}\Xi_{\nu,c}(z)$ belongs also to the Laguerre-P\'olya class $\mathcal{LP}.$ Hence the zeros of the function ${}_{k}\Xi_{\nu,c}(z)$ are all real. Thus, ${}_{k}\Xi_{\nu,c}(z)$ can be represented as the inifinite product
	\begin{equation}\label{BoundsforStar3}
	{}_{k}\Xi_{\nu,c}(z)=\frac{\nu\left( \frac{z}{2}\right) ^{\frac{\nu}{k}-1}}{2k\Gamma_{k}(\nu+k)}\prod_{n\geq 1}\left(1-\frac{z^2}{{}_{k}\omega_{\nu,c,n}'^2}\right).
	\end{equation}
	Logarithmic differentiation of both sides of Eqn. \eqref{BoundsforStar3} for $\left| z\right|< {}_{k}\omega_{\nu,c,1}'$ yields 
	\begin{equation}\label{BoundsforStar4}
	\frac{z	{}_{k}\Xi_{\nu,c}'(z)}{	{}_{k}\Xi_{\nu,c}(z)}-\left( \frac{\nu}{k}-1\right) =-2\sum_{n \geq 1}\frac{z^2}{{}_{k}\omega_{\nu,c,n}'^2-z^2}=-2\sum_{n \geq 1}\sum_{m\geq0}\frac{z^{2m+2}}{{}_{k}\omega_{\nu,c,n}'^{2m+2}}=-2\sum_{m\geq0}\sigma_{m+1}z^{2m+2},
	\end{equation}
	where $\sigma_{m}=\sum_{n \geq 1}{}_k{}\omega_{\nu,c,n}'^{-2m}.$ On the other hand, with the aid of Eqns. \eqref{BounsforStar1} and \eqref{BounsforStar2} we find
	\begin{equation}\label{BoundsforStar5}
	\frac{z	{}_{k}\Xi_{\nu,c}'(z)}{	{}_{k}\Xi_{\nu,c}(z)}=\sum_{n\geq 0}a_n z^{2n} \bigg/ \sum_{n \geq 0}b_n z^{2n}
	\end{equation}
	where 
	\[ a_n= \frac{(-c)^n (2n+\frac{\nu}{k})(2n+\frac{\nu}{k}-1)}{2^{2n}n!\Gamma_{k}(nk+\nu+k)} \text{ \ and \ } b_n=\frac{(-c)^n (2n+\frac{\nu}{k})}{2^{2n}n!\Gamma_{k}(nk+\nu+k)}.\]
	By comparing the coefficients of Eqns. \eqref{BoundsforStar4} and \eqref{BoundsforStar5} we obtain
	\[ \sigma_{1}=\frac{c(\nu+2k)}{4\nu(\nu+k)} \text{ \ and \ } \sigma_{2}=\frac{c^2\left( (\nu+2k)^3-k\nu(\nu+k)(\nu+4k)\right) }{16k\nu(\nu+k)^2(\nu+2k)} .\]
	By using the Euler-Rayleigh inequalities
	\[ \sigma_{m}^{-\frac{1}{m}}<{}_{k}\omega_{\nu,c,1}'^{2}<\frac{ \sigma_{m}}{ \sigma_{m+1}}\]
	for $m=1$ we obtain the inequalities of the first part of the theorem.
	\item [\bf b.] From \cite{BKS} and \cite{BOS} we say that the radius of starlikeness of the function ${}_{k}g_{\nu,c}(z)$ is the first positive zero of its derivative. We can draw conclusion from \Cref{kBesselLemma} that the zeros
	\[ {}_{k}g_{\nu,c}(z)=\Gamma_{k}(\nu+k)\sum_{n \geq 0}\frac{(-c)^n}{n!2^{2n}\Gamma_{k}(nk+\nu+k)}z^{2n+1}  \]
	are all real for $k>0$, $c>0$ and $\nu>0.$ Consequently, this function belongs to the Laguerre-P\'olya class. Since the Laguerre-P\'olya class $\mathcal{LP}$ is closed under differentiation, we deduce that ${}_{k}g_{\nu,c}'(z)$ belongs also to the  Laguerre-P\'olya class and hence all of its zeros are real. Now, we consider the entire function
	\[{}_{k}\Lambda_{\nu,c}(z)={}_{k}g_{\nu,c}'(2\sqrt{z}) =\Gamma_{k}(\nu+k)\sum_{n \geq 0}\frac{(-c)^n (2n+1)}{n!\Gamma_{k}(nk+\nu+k)}z^{n}.\]
	We shall show that all the zeros of the function ${}_{k}\Lambda_{\nu,c}(z)$ are real and positive. For this we note that
	\[ {}_{k}\phi_{\nu,c}(z)=\frac{\Gamma_{k}(\nu+k)(2z+1)}{\Gamma_{k}(kz+\nu+k)} \]
	is entire function of growth order $1$ and this assume real values along the real axis and possess only negative zeros if $k>0$, $c>0$ and $\nu>0.$ Therefore in light of Laguerre's lemma stated in \cite[Lemma 1, pp. 2208]{BS18} we get that the entire function
	\[ {}_{k}\gamma_{\nu,c}(z)=\Gamma_{k}(\nu+k)\sum_{n \geq 0}\frac{ (2n+1)}{n!\Gamma_{k}(nk+\nu+k)}z^{n} \]
	also has only real and negative zeros. This means that for $k>0$, $c>0$ and $\nu>0$ ${}_{k}\gamma_{\nu,c}(-cz)$ has real and positive zeros. That is, ${}_{k}\Lambda_{\nu,c}(z)$ has real and positive zeros. Suppose that ${}_{k}\beta_{\nu,c,n}$'s are the zeros of the function ${}_{k}\Lambda_{\nu,c}(z).$ Thus, since the function $z\mapsto{}_{k}\Lambda_{\nu,c}(z)$ has growth order $\frac{1}{2}$ it can be represented by the infinite product
	\begin{equation}\label{BoundsforStar6}
	{}_{k}\Lambda_{\nu,c}(z)=\prod_{n\geq1}\left(1-\frac{z}{{}_{k}\beta_{\nu,c,n}} \right),
	\end{equation}
	where ${}_{k}\beta_{\nu,c,n}>0$ for each $n\in \mathbb{N}.$ 	Logarithmic differentiation of both sides of Eqn. \eqref{BoundsforStar6} yields
	\begin{align}\label{BoundsforStar7}
	\frac{{}_{k}\Lambda_{\nu,c}'(z)}{{}_{k}\Lambda_{\nu,c}(z)}&=-\sum_{n \geq 1}\frac{1}{{}_{k}\beta_{\nu,c,n}-z}=-\sum_{n \geq 1}\sum_{m\geq 0} \frac{z^m}{{}_{q}\beta_{\nu,c,n}^{m+1}} \nonumber \\
	&=-\sum_{m\geq 0}\sum_{n \geq 1} \frac{z^m}{{}_{q}\beta_{\nu,c,n}^{m+1}}=-\sum_{m\geq 0} \ell_{m+1}z^m, \text{ \ \ } \left| z\right| <{}_{k}\beta_{\nu,c,1},
	\end{align}
	where $ \ell_{m}=\sum_{n \geq 1}{}_{k}\beta_{\nu,c,n}^{-m}.$ Moreover, taking into account fact that
	\[{}_{k}\Lambda_{\nu,c}'(z)=\Gamma_{k}(\nu+k)\sum_{n \geq 0}\frac{(-c)^{n+1}(2n+3)}{n!\Gamma_{k}((n+1)k+\nu+k)}z^{n}\]
	we get 
	\begin{equation}\label{BoundsforStar8}
	\frac{{}_{k}\Lambda_{\nu,c}'(z)}{{}_{k}\Lambda_{\nu,c}(z)}=\sum_{n \geq 0}c_n z^n\bigg/ \sum_{n \geq 0}d_n z^n,
	\end{equation}
	where
	\[c_n=\frac{(-c)^{n+1}(2n+3)}{n!\Gamma_{k}((n+1)k+\nu+k)} \text{ \ and \ } d_n=\frac{(-c)^n (2n+1)}{n!\Gamma_{k}(nk+\nu+k)}.\]
	By comparing the coefficients of Eqns. \eqref{BoundsforStar7} and \eqref{BoundsforStar8} we arrive at
	\[\ell_1=\frac{3c}{\nu+k} \text{ \ and \ } \ell_2=\frac{c^2\left(  4\nu+13k\right)  }{(\nu+k)^2 (\nu+2k)} .\]
	By using the Euler-Rayleigh inequalities $\ell_{m}^{-\frac{1}{m}}<{}_{k}\beta_{\nu,c,1}<\frac{\ell_{m}}{\ell_{m+1}}$ we obtain the next inequalities for $2\sqrt{{}_{k}\beta_{\nu,c,1}},$ that is,
	\[2\sqrt{\frac{\nu+3}{3c}}<r^{\star}({}_{k}g_{\nu,c})<2\sqrt{\frac{3(\nu+k)(\nu+2k)}{c\left( 4\nu+13k\right) }}.\]
	\item [\bf c.] Consider the entire function
	\begin{equation}\label{BoundsforStar9}
	{}_{k}\Upsilon_{\nu,c}(z)={}_{k}h_{\nu,c}'(4z)=\Gamma_{k}(\nu+k)\sum_{n\geq 0}\frac{(-c)^n (n+1)}{ n!\Gamma_{k}(kn+\nu+k)}z^{n} .
	\end{equation}
	With the aid of \Cref{kBesselLemma}, it is possible to prove the reality of the zeros of the function ${}_{k}h_{\nu,c}(z).$ This means that ${}_{k}h_{\nu,c}(z)$ belongs to the Laguerre-P\'olya class $\mathcal{LP}$. Consequently, the function ${}_{k}h_{\nu,c}'(z)$ belongs also to the Laguerre-P\'olya class $\mathcal{LP}$ and has only real zeros. It is obvious that this is also valid for the function ${}_{k}\Upsilon_{\nu,c}(z).$ Moreover, by means of Laguerre's lemma stated in \cite[Lemma 1, pp. 2208]{BS18} we deduce that the function ${}_{k}\Upsilon_{\nu,c}(z)$ has only positive real zeros and has growth order $\frac{1}{2},$ and thus ${}_{k}\Upsilon_{\nu,c}(z)$ can be represented by the infinite product
	\begin{equation}\label{BoundsforStar10}
	{}_{k}\Upsilon_{\nu,c}(z)=\prod_{n\geq1}\left(1-\frac{z}{{}_{k}\varsigma_{\nu,c,n}} \right),
	\end{equation}
	where ${}_{k}\varsigma_{\nu,c,n}$'s are positive zeros of the function ${}_{k}\Upsilon_{\nu,c}(z).$ Logarithmic differentiation of both sides of Eqn. \eqref{BoundsforStar10} gives
	\begin{equation}\label{BoundsforStar11}
	\frac{{}_{k}\Upsilon_{\nu,c}'(z)}{{}_{k}\Upsilon_{\nu,c}(z)}=-\sum_{n \geq 1}\frac{1}{{}_{k}\varsigma_{\nu,c,n}-z} =-\sum_{n \geq 1}\sum_{m\geq 0}\frac{z^m}{{}_{k}\varsigma_{\nu,c,n}^{m+1}}=-\sum_{m\geq 0}\kappa_{m+1}z^m, \text{ \ \ }\left| z\right|<{}_{k}\varsigma_{\nu,c,1}
	\end{equation}
	where $\kappa_{m}=\sum_{n \geq 1}{}_{k}\varsigma_{\nu,c,n}^{-m}.$ On the other hand, with the aid of Eqn.\eqref{BoundsforStar9} we have
	\begin{equation}\label{BoundsforStar12}
	\frac{	{}_{k}\Upsilon_{\nu,c}'(z)}{	{}_{k}\Upsilon_{\nu,c}(z)}= \sum_{n\geq 0}\frac{(-c)^{n+1} (n+2)}{ n!\Gamma_{k}(k(n+1)+\nu+k)}z^{n} \bigg/ \sum_{n\geq 0}\frac{(-c)^n (n+1)}{ n!\Gamma_{k}(kn+\nu+k)}z^{n}.
	\end{equation}
	With the help of Eqns. \eqref{BoundsforStar11} and \eqref{BoundsforStar12} we can express the Euler-Rayleigh sums in terms of $k,$ $\nu,$ $c$ and by using the Euler-Rayleigh inequalities $\kappa_{m}^{-\frac{1}{m}}<{}_{k}\varsigma_{\nu,c,1}<\frac{\kappa_{m}}{\kappa_{m+1}}$ we obtain the inequalities for $4{}_{k}\varsigma_{\nu,c,1}$ for $k>0$, $\nu>0$, $c>0$ and $m\in \mathbb{N}$ 
	\[4\kappa_{m}^{-\frac{1}{m}}<r^{\star}({}_{k}h_{\nu,c})<4\frac{\kappa_{m}}{\kappa_{m+1}}.\]	
	Since
	\[\kappa_{1}=\frac{2c}{\nu+k}   \text{ \ and \ } \kappa_{2}= \frac{4c^2}{(\nu+k)^2}-\frac{3c^2}{(\nu+k)(\nu+2k)}\]
	in particular, for $m=1$ from the above Euler-Rayleigh inequalities we have the next inequality for $4{}_{k}\varsigma_{\nu,c,1}$, that is,
	\[\frac{2(\nu+k)}{c}<r^{\star}({}_{k}h_{\nu,c})<\frac{8(\nu+k)(\nu+2k)}{c(\nu+5k)}.\]
\end{itemize}
This completes the proof of the theorem.
\end{proof}
\begin{remark}
It is obvious that our main results which are presented in \Cref{MainTheorem2} when we take $c=1$ and $k=1,$ coincide with the inequalities in \cite[Thms. 1 and 2]{ABY}.
\end{remark}

\subsection{The radii of  convexity of order $\alpha$ of functions ${}_{k}f_{\nu,c}$, ${}_{k}g_{\nu,c}$, and ${}_{k}h_{\nu,c}$} In this section we aim to determine the radii of convexity of order $\alpha$ of the normalized generalized $k-$Bessel functions and to find tight lower and upper bounds for the radius of convexity of order zero of these functions with the help of Euler-Rayleigh inequalities.

\begin{theorem}\label{MainTheorem3}
Let $k>0$, $c>0,$ $\nu>0$ and $\alpha\in\left[ 0,1\right) $.
\begin{itemize}
	\item [\bf a.] The radius of convexity of order $\alpha$ of the function ${}_{k}f_{\nu,c}$ is the smallest root of the equation
	\[1+r\frac{{}_{k}W_{\nu,c}''(r)}{{}_{k}W_{\nu,c}'(r)}+\left(\frac{k}{\nu}-1\right)r\frac{{}_{k}W_{\nu,c}'(r)}{{}_{k}W_{\nu,c}(r)}=\alpha.\]
	
	\item [\bf b.] The radius of convexity of order $\alpha$ of the function ${}_{k}g_{\nu,c}$ is the smallest root of the equation
	\[1+r\frac{{}_{k}g_{\nu,c}''(r)}{{}_{k}g_{\nu,c}'(r)}=\alpha.\]
	
	\item [\bf c.] The radius of convexity of order $\alpha$ of the function ${}_{k}h_{\nu,c}$ is the smallest root of the equation
	\[1+r\frac{{}_{k}h_{\nu,c}''(r)}{{}_{k}h_{\nu,c}'(r)}=\alpha.\]
\end{itemize}
\end{theorem}
\begin{proof}
\begin{itemize}
	\item [\bf a.]  It is easy to verify that
	\[1+z\frac{{}_{k}f_{\nu,c}''(z)}{{}_{k}f_{\nu,c}'(z)}=1+z\frac{{}_{k}W_{\nu,c}''(z)}{{}_{k}W_{\nu,c}'(z)}+\left(\frac{k}{\nu}-1\right)z\frac{{}_{k}W_{\nu,c}'(z)}{{}_{k}W_{\nu,c}(z)}.\]
	Let denote of  the $n$th positive roots of ${}_{k}W_{\nu,c}(z)$ and ${}_{k}W_{\nu,c}'(z)$ by ${}_{k}\omega_{\nu,c,n}$ and ${}_{k}\omega_{\nu,c,n}^{\prime}$, respectively. From Eqn. \eqref{BoundsforStar3} we have the following infinite product representation
	\begin{equation}\label{EqConvex1}
	 {}_{k}W_{\nu,c}'(z)=\frac{\nu\left( \frac{z}{2}\right) ^{\frac{\nu}{k}-1}}{2k\Gamma_{k}(\nu+k)}\prod_{n\geq 1}\left(1-\frac{z^2}{{}_{k}\omega_{\nu,c,n}'^2}\right).
	\end{equation}
	Then  Logarithmic differentiation of Eqn. \eqref{InfiniteProductkBessel} stated in \Cref{kBesselLemma} and Eqn. \eqref{EqConvex1} leads to
	\[1+z\frac{{}_{k}f_{\nu,c}''(z)}{{}_{k}f_{\nu,c}'(z)}=1-\left(\frac{k}{\nu}-1\right)\sum_{n\geq1}\frac{2z^2}{{}_{k}\omega_{\nu,c,n}^2-z^2}-\sum_{n \geq 1}\frac{2z^2}{{}_{k}\omega_{\nu,c,n}'^2-z^2}. \]
	We will prove the theorem in two steps. First suppose $k>\nu$ By using the inequality \eqref{Section1Eq.2} we have
	\begin{equation}\label{EqConvex2}
	\real\left(1+z\frac{{}_{k}f_{\nu,c}''(z)}{{}_{k}f_{\nu,c}'(z)}\right)\geq1-\left(\frac{k}{\nu}-1\right)\sum_{n\geq1}\frac{2r^2}{{}_{k}\omega_{\nu,c,n}^2-r^2}-\sum_{n \geq 1}\frac{2r^2}{{}_{k}\omega_{\nu,c,n}'^2-r^2},
	\end{equation}
	where $\left| z\right|=r. $ Moreover, in light of the following inequality stated in \cite[Lemma 2.1]{BSz14}
	\[\alpha\real\left(\frac{z}{a-z} \right)-\real\left(\frac{z}{b-z} \right)\geq \alpha\frac{\left| z\right| }{a-\left| z\right|}-\frac{\left| z\right|}{b-\left| z\right|},\]
	where $a>b>0,$ $\alpha \in\left[0,1 \right] ,$ $z\in \mathbb{C}$ such that $\left| z\right|<b,$ we obtain that Eqn. \eqref{EqConvex2} also valid when $k<\nu$ for all $\mathbb{D}_{{}_{k}\omega_{\nu,c,1}'}.$ Here we used that the zeros ${}_{k}\omega_{\nu,c,n}$ and ${}_{k}\omega_{\nu,c,n}'$ interlace according to \Cref{kBesselLemma}.  Now, the above deduced inequality implies for $r\in\left(0, {}_{k}\omega_{\nu,c,1}'\right)$
	\[\inf_{z\in\mathbb{D}_{r}}\left\lbrace\real\left(1+z\frac{{}_{k}f_{\nu,c}''(z)}{{}_{k}f_{\nu,c}'(z)}\right)\right\rbrace=1+r\frac{{}_{k}f_{\nu,c}''(r)}{{}_{k}f_{\nu,c}'(r)}.\]
	Now we deal with the function ${}_{k}\vartheta_{\nu,c}:\left(0, {}_{k}\omega_{\nu,c,1}' \right) \rightarrow \mathbb{R},$ defined by
	\[{}_{k}\vartheta_{\nu,c}(r)=1+r\frac{{}_{k}f_{\nu,c}''(r)}{{}_{k}f_{\nu,c}'(r)}.\]
	The function is strictly decreasing since
	\begin{align*}
	{}_{k}\vartheta_{\nu,c}'(r)&=-\left(\frac{k}{\nu}-1 \right)\sum_{n \geq 1}\frac{4r{}_{k}\omega_{\nu,c,n}^2}{\left({}_{k}\omega_{\nu,c,n}^2-r^2\right)^2}-\sum_{n\geq1}\frac{4r{}_{k}\omega_{\nu,c,n}'^2}{\left({}_{k}\omega_{\nu,c,n}'^2-r^2\right)^2}\\
	&<\sum_{n\geq1}\frac{4r{}_{k}\omega_{\nu,c,n}^2}{\left({}_{k}\omega_{\nu,c,n}^2-r^2\right)^2}-\sum_{n\geq1}\frac{4r{}_{k}\omega_{\nu,c,n}'^2}{\left({}_{k}\omega_{\nu,c,n}'^2-r^2\right)^2}<0
	\end{align*}
	for $ r\in\left(0, {}_{k}\omega_{\nu,c,1}'\right) ,$ where we used again the interlacing property of the zeros stated in \Cref{kBesselLemma}. Also, taking into consideration that $\lim_{r \searrow 0}{}_{k}\vartheta_{\nu,c}(r)=1-\alpha>0,$ $\lim_{r\nearrow {}_{k}\omega_{\nu,c,1}}{}_{k}\vartheta_{\nu,c}(r)=-\infty$ that means that for $z\in \mathbb{D}(0,r_1)$ we have
	\[\real\left(1+z\frac{{}_{k}f_{\nu,c}''(z)}{{}_{k}f_{\nu,c}'(z)} \right) >\alpha,\]
	if and only if $r_1$ is the unique root of
	\[1+z\frac{{}_{k}f_{\nu,c}''(z)}{{}_{k}f_{\nu,c}'(z)}=\alpha,\]
	situated in $\left(0, {}_{k}\omega_{\nu,c,1}' \right). $
	\item [\bf b.] We know that the function ${}_{k}g_{\nu,c}'$ belongs to Laguerre-P\'olya class $\mathcal{LP}$ and has only real zeros. Suppose that ${}_{k}\eta_{\nu,c,n}$'s are the real zeros of the function ${}_{k}g_{\nu,c}'.$ Thus since the function $z\mapsto {}_{k}g_{\nu,c}' $ has growth order $\rho({}_{k}g_{\nu,c}')=\frac{1}{2}$ it can be represented by the infinite product
	\[{}_{k}g_{\nu,c}'(z)=\prod_{n\geq1}\left(1-\frac{z^2}{{}_{k}\eta_{\nu,c,n}^{2}} \right). \]
	Now, taking logarithmic derivatives on both sides, we obtain
	\[1+z\frac{{}_{k}g_{\nu,c}''(z)}{{}_{k}g_{\nu,c}'(z)}=1-2\sum_{n \geq 1}\frac{z^2}{{}_{k}\eta_{\nu,c,n}^2-z^2}.\]
	Application of the inequality Eqn. \eqref{Section1Eq.2} implies that
	\[\real\left(1+z\frac{{}_{k}g_{\nu,c}''(z)}{{}_{k}g_{\nu,c}'(z)} \right) \geq 1- 2\sum_{n \geq 1}\frac{r^2}{{}_{k}\eta_{\nu,c,n}^{2}-r^2}, \]
	where $\left|z \right|=r.$ Thus, $r\in\left(0,{}_{k}\eta_{\nu,c,1} \right),$ we get
	\[ \inf_{z\in \mathbb{D}_{r}}\left\lbrace \real\left(1+z\frac{{}_{k}g_{\nu,c}''(z)}{{}_{k}g_{\nu,c}'(z)} \right)\right\rbrace=1- 2\sum_{n \geq 1}\frac{r^2}{{}_{k}\eta_{\nu,c,n}^{2}-r^{2}}=1+r\frac{{}_{k}g_{\nu,c}''(r)}{{}_{k}g_{\nu,c}'(r)}.\]
	The function ${}_{k}\phi_{\nu,c}:\left(0,{}_{k}\eta_{\nu,c,1} \right)\rightarrow \mathbb{R}$ defined by
	\[{}_{k}\phi_{\nu,c}(r)=1-\alpha+r\frac{{}_{k}g_{\nu,c}''(r)}{{}_{k}g_{\nu,c}'(r)} \]
	is strictly decreasing and take the limits \(\lim_{r\nearrow {}_{k}\eta_{\nu,c,1}}{}_{k}\phi_{\nu,c}(r)=-\infty\) and \(\lim_{r \searrow 0}{}_{k}\phi_{\nu,c}(r)=1-\alpha, \) which means that for $z\in \mathbb{D}_{r_2}$ we get
	\[\real\left(1+z\frac{{}_{k}g_{\nu,c}''(z)}{{}_{k}g_{\nu,c}'(z)} \right) >\alpha,\]
	if and only if $r_1$ is the unique root of
	\[1+z\frac{{}_{k}g_{\nu,c}''(z)}{{}_{k}g_{\nu,c}'(z)}=\alpha,\]
	situated in $\left(0, {}_{k}\eta_{\nu,c,1}\right). $
	\item [\bf c.]  We know that the function $z\mapsto{}_{k}h_{\nu,c}$ belongs to the Laguerre-P\'olya class $\mathcal{LP},$ and consequently we get the function $z\mapsto{}_{k}h_{\nu,c}'$ belongs also to the  Laguerre-P\'olya class $\mathcal{LP}.$ Hence the zeros of the function  $z\mapsto {}_{k}h_{\nu,c}'$ are all real. Moreover, in light of Laguerre's lemma stated in \cite{BS18}, we say that the function ${}_{k}h_{\nu,c}'$ has only positive real zeros.  Let ${}_{k}\theta_{\nu,c,n}$ be the $n$th positive zero of the function $z\mapsto {}_{k}h_{\nu,c}'.$ Since the function $z\mapsto {}_{k}h_{\nu,c}' $ has growth order $\rho({}_{k}h_{\nu,c}')=\frac{1}{2}$ the next infinite product representation is valid
	\begin{equation}
	 {}_{k}h_{\nu,c}'(z)=\prod_{n\geq 1}\left(1-\frac{z}{{}_{k}\theta_{\nu,c,n}} \right).
	\end{equation}
	Let $r \in (0, {}_{k}\theta_{\nu,c,1})$ be a fixed number. Because of the minimum principle for harmonic functions and inequality \eqref{Section1Eq.2} for $z\in\mathbb{D}_r$ we have 
	\[\real\left(1+z\frac{{}_{k}h_{\nu,c}''(z)}{{}_{k}h_{\nu,c}'(z)}\right)=1-\real\left(\sum_{n\geq1}\frac{z}{{}_{k}\theta_{\nu,c,n}-z}\right) \geq 1+r\frac{{}_{k}h_{\nu,c,}''(r)}{{}_{k}h_{\nu,c}'(r)}. \]
	Consequently, it follows that
	\[\inf_{z\in\mathbb{D}_{r}}\left\lbrace\real\left(1+z\frac{{}_{k}h_{\nu,c,}''(z)}{{}_{k}h_{\nu,c}'(z)}\right)\right\rbrace=1+r\frac{{}_{k}h_{\nu,c}''(r)}{{}_{k}h_{\nu,c}'(r)}. \]
	Now, let $r_3$ be the smallest positive root of the equation
	\begin{equation}\label{EqConvex3}
	1+r\frac{{}_{k}h_{\nu,c,}''(r)}{{}_{k}h_{\nu,c}'(r)}=\alpha.
	\end{equation}
	For $z\in\mathbb{D}_{r_3}$ we have
	\[\real\left(1+r\frac{{}_{k}h_{\nu,c}''(r)}{{}_{k}h_{\nu,c}'(r)}\right)>\alpha.\]
	Now, we need to show that equation \eqref{EqConvex3} has a unique root in $(0, {}_{k}\theta_{\nu,c,1}).$ The function ${}_{k}\Theta_{\nu,c,\delta}:(0,{}_{k}\theta_{\nu,c,1})\rightarrow \mathbb{R}$ defined by
	\[{}_{k}\Theta_{\nu,c}(r)=1-\alpha+r\frac{{}_{k}h_{\nu,c}''(r)}{{}_{k}h_{\nu,c}'(r)}\]
	is strictly decreasing and 
	\[\lim_{r\nearrow{}_{k}\theta_{\nu,c,1}}{}_{k}\Theta_{\nu,c}(r)=-\infty, \text{ \ \ } \lim_{r \searrow 0}{}_{k}\Theta_{\nu,c}(r)=1-\alpha.\]
	Consequently, the equation
	\[1+z\frac{{}_{k}h_{\nu,c}''(z)}{{}_{k}h_{\nu,c}'(z)}=\alpha\]
	has a unique root $r_3$ in $(0,{}_{k}\theta_{\nu,c,1}).$
\end{itemize}
This completes the proof of the theorem.
\end{proof}
\begin{remark}
It is clear that  our main results which are given in  \Cref{MainTheorem3} when we choose $c=1$ and $k=1,$ correspond to the results given in \cite[Thms. 1, 2, and 3]{BSz14}.
\end{remark}
Finally, we give some tight lower and upper bounds for the radii of convexity of the functions ${}_{k}g_{\nu,c}$ and ${}_{k}h_{\nu,c}$.
\begin{theorem}\label{MainTheorem4}
Let $k>0$, $c>0,$ $\nu>0.$
\begin{itemize}
	\item [\bf a.] The radius of convexity $r^{c}({}_{k}g_{\nu,c})$ of the function 
	\[z\mapsto{}_{k}g_{\nu,c}(z)=2^{\frac{\nu}{k}}\Gamma_{k}(\nu+k)z^{1-\frac{\nu}{k}}{}_{k}W_{\nu,c}(z),\]
	is the smallest root of the $\left(z{}_{k}g_{\nu,c}'\right)^{\prime}=0$ and satisfies the following inequality
	\[\frac{2}{3}\sqrt{\frac{\nu+k}{c}}<r^{c}({}_{k}g_{\nu,c})<6\sqrt{\frac{(\nu+k)(\nu+2k)}{c(57\nu+137k)}} .\]
	\item [\bf b.] The radius of convexity $r^{c}({}_{k}h_{\nu,c})$ of the function 
	\[z\mapsto{}_{k}h_{\nu,c}(z)=2^{\frac{\nu}{k}}\Gamma_{k}(\nu+k)z^{1-\frac{\nu}{2k}}{}_{k}W_{\nu,c}(\sqrt{z}),\]
	is the smallest root of the $\left(z{}_{k}h_{\nu,c}'\right)^{\prime}=0$ and satisfies the following inequality
	\[ \frac{\nu+k}{c}< r^{c}({}_{k}h_{\nu,c})<\frac{16(\nu+k)(\nu+2k)}{c(7\nu+23k)} .\]
\end{itemize}
\end{theorem}
\begin{proof}
\begin{itemize}
	\item [\bf a.] In order to prove our main result we will make use of the Alexander’s duality theorem which has a very simple proof based on the characterization of starlike and convex functions in the unit disc. By means of this theorem one can deduce that the function ${}_{k}g_{\nu,c}(z)$ is convex if and only if $z\mapsto \left( z{}_{k}g_{\nu,c}\right)^{\prime}$ is starlike. From the studies in \cite{BKS,BOS} we know that the smallest positive zero of $z\mapsto\left(z{}_{k}g_{\nu,c}'\right)^{\prime}$ is the radius of starlikeness of $z{}_{k}g_{\nu,c}'(z).$ That is why the radius of convexity $r^{c}({}_{k}g_{\nu,c})$ is the smallest positive root of the equation $\left(z{}_{k}g_{\nu,c}'\right)^{\prime}=0.$ Now, by using  Eqns. \eqref{kBessel1} and \eqref{BounsforStar1}, which are the infinite series representations of the generalized $k-$Bessel function and its derivative, we have
	\[\left(z{}_{k}g_{\nu,c}'(z) \right)'=\Gamma_{k}(\nu+k)\sum_{n\geq 0}\frac{(-c)^{n}(2n+1)^{2}}{2^{2n}n!\Gamma_{k}(kn+\nu+k)}z^{2n}.\]
	Moreover, if we take $2\sqrt{z}$ instead of $z$ in the above equality, it is obvious that
	\begin{equation}\label{EqConvex4}
	{}_{k}\Delta_{\nu,c}(z)=\Gamma_{k}(\nu+k)\sum_{n\geq0}\frac{(-c)^{n}(2n+1)^{2}}{n!\Gamma_{k}(kn+\nu+k)}z^{n}.
	\end{equation}
	Taking into account facts that the function ${}_{k}g_{\nu,c}$ belongs to the Laguerre-P\'olya class of entire functions and that the class $\mathcal{LP}$ is closed under differentiation, it is easy to deduce that the function ${}_{k}\Delta_{\nu,c}$ belongs also to the Laguerre-P\'olya class. As a result, the function ${}_{k}\Delta_{\nu,c}$ is an entire function that has only real zeros. Furthermore, with the help of the Laguerre's theorem stated in \cite[Lemma 1 p.p 2208]{BS18} we deduce that the function ${}_{k}\Delta_{\nu,c}$ has real and positive zeros.  Suppose that ${}_{k}\varrho_{\nu,c,n}$'s are the positive zeros of the function ${}_{k}\Delta_{\nu,c}.$ Then the function ${}_{k}\Delta_{\nu,c}$ has the infinite product representation as follows:
	\begin{equation}\label{EqConvex5}
	{}_{k}\Delta_{\nu,c}(z)=\prod_{n\geq1}\left(1-\frac{z}{{}_{k}\varrho_{\nu,c,n}}\right). 
	\end{equation}
	Now, taking logarithmic derivatives on both sides, we arrive at
	\begin{equation}\label{EqConvex6}
	\frac{{}_{k}\Delta_{\nu,c}'(z)}{{}_{k}\Delta_{\nu,c}(z)}=-\sum_{n \geq 1}\frac{1}{{}_{k}\varrho_{\nu,c,n}-z}=-\sum_{m\geq 0}\mu_{m+1}z^m, \text{ \ \ } \left|z \right| < {}_{k}\varrho_{\nu,c,1},
	\end{equation}
	where $\mu_{m}=\sum_{n \geq 1}{}_{k}\varrho_{\nu,c,n}^{-m}.$ On the other hand, by using Eqn. \eqref{EqConvex4}, we get
	\begin{equation}\label{EqConvex7}
	\frac{{}_{k}\Delta_{\nu,c}'(z)}{{}_{k}\Delta_{\nu,c}(z)}=\sum_{n \geq 0}\frac{(-c)^{n+1}(2n+3)^2}{n!\Gamma_{k}(k(n+1)+\nu+k)}z^n \bigg/ \sum_{n \geq 0}\frac{(-c)^n (2n+1)^2}{n!\Gamma_{k}(kn+\nu+k)}z^n.
	\end{equation}
	By comparing the coefficients of Eqns. \eqref{EqConvex6} and \eqref{EqConvex7} we obtain
	\[\mu_{1}=\frac{9c}{\nu+k} \text{ \ and\ } \mu_{2}=\frac{81c^2}{(\nu
	+k)^2}-\frac{25c^2}{(\nu+k)(\nu+2k)} \]
	and by considering the Euler-Rayleigh inequalities \(\mu_{m}^{-\frac{1}{m}}<{}_{k}\varrho_{\nu,c,1}<\frac{\mu_{m}}{\mu_{m+1}} \) we have the inequalities for \(2\sqrt{{}_{k}\varrho_{\nu,c,1}} \)
	\[\frac{2}{3}\sqrt{\frac{\nu+k}{c}}<r^{c}({}_{k}g_{\nu,c})<6\sqrt{\frac{(\nu+k)(\nu+2k)}{c(57\nu+137k)}}. \]
	\item [\bf b.] In light of explanations which are presented in the proof of the first part of the theorem one can deduce that the radius of convexity $r^{c}({}_{k}h_{\nu,c})$ is the smallest positive root of the equation $\left(z{}_{k}h_{\nu,c}'(z)\right)^{\prime}=0.$ Upon some simple calculation, we obtain
	\begin{equation}\label{EqConvex8}
	{}_{k}\lambda_{\nu,c}(z)=\left(z{}_{k}h_{\nu,c}'(z)\right)^{\prime}=1+\Gamma_{k}(\nu+k)\sum_{n\geq1}\frac{(-c)^{n}(n+1)^{2}}{2^{2n}n!\Gamma_{k}(kn+\nu+k)}z^{n}.
	\end{equation}
	With the aid of facts that the function ${}_{k}h_{\nu,c}$ belongs to the Laguerre-P\'olya class of entire functions $\mathcal{LP}$ and that the class  $\mathcal{LP}$ is closed under differentiation, we say that the function ${}_{k}\lambda_{\nu,c}$ is also in the Laguerre-P\'olya class. This means that the zeros of the function ${}_{k}\lambda_{\nu,c}$ are all real. Now, we shall show the positivity of the zeros of ${}_{k}\lambda_{\nu,c}.$ For this we note that
	\[{}_{k}a_{\nu,c}(4z)=\frac{1}{\Gamma_{k}(kz+\nu+k)} \] 
	which is entire function of growth order $1$ and it assumes real values along the real axis and possess only negative zeros if $\nu>0$ and $k>0.$ Therefore in light of Laguerre's Lemma stated in \cite{BS18} we obtain that the entire function
	\[{}_{k}u_{\nu,c}(z)=\sum_{n \geq 0}\frac{z^n}{n!\Gamma_{k}(kn+\nu+k)} \]
	also has real and negative zeros. Hence ${}_{k}u_{\nu,c}(-cz)$ has real and positive zeros. That is, ${}_{k}\lambda_{\nu,c}(z)$ has real and positive zeros. Suppose that ${}_{k}\tau_{\nu,c,n}$'s are the positive zeros of the function ${}_{k}\lambda_{\nu,c}.$ Then the infinite product representation of the function ${}_{k}\lambda_{\nu,c}(z)$ can be stated as
	\begin{equation}\label{EqConvex9}
	{}_{k}\lambda_{\nu,c}(z)=\prod_{n\geq1}\left(1-\frac{z}{{}_{k}\tau_{\nu,c,n}} \right).
	\end{equation} 
	Now, taking the logarithmic derivatives on both sides of Eqn. \eqref{EqConvex9}, we have 
	\begin{equation}\label{EqConvex10}
	\frac{{}_{k}\lambda_{\nu,c}'(z)}{{}_{k}\lambda_{\nu,c}(z)}=-\sum_{m\geq 0}\upsilon_{m+1}z^{m}, \text{ \ \ } \left|z \right|<{}_{k}\tau_{\nu,c,1},
	\end{equation}
	where $\upsilon_{m}=\sum_{n \geq 1} {}_{k}\tau_{\nu,c,n}^{-m}.$ Also, by making use of Eqn. \eqref{EqConvex8} and its derivative, we get
	\begin{equation}\label{EqConvex11}
	\frac{{}_{k}\lambda_{\nu,c}'(z)}{{}_{k}\lambda_{\nu,c}(z)}=\sum_{n \geq 0}\frac{(-c)^{n+1}(n+2)^2}{2^{2n+2}n!\Gamma_{k}(k(n+1)+\nu+k)}z^n\bigg/ \sum_{n \geq 0}\frac{(-c)^n(n+1)^2}{2^{2n}n!\Gamma_{k}(kn+\nu+k)}z^n
	\end{equation}
	By comparing the coefficients of Eqns. \eqref{EqConvex10} and \eqref{EqConvex11} we arrive at
	\[\upsilon_{1}=\frac{c}{\nu+k} \text{ \ and\ } \upsilon_{2}=\frac{c^2\left(7\nu+23k \right) }{16(\nu+k)(\nu+2k)}.\]
	By considering the Euler-Rayleigh inequalities $\upsilon_{m}^{-\frac{1}{m}}<{}_{k}\tau_{\nu,c,n}<\frac{\upsilon_{m}}{\upsilon_{m+1}}$ for $m=1$ we have
	\[\frac{\nu+k}{c}<r^{c}({}_{k}h_{\nu,c})<\frac{16(\nu+k)(\nu+2k)}{c(7\nu+23k)} .\]
\end{itemize}
This completes the proof of the theorem.
\end{proof}
\begin{remark}
It is clear that  our main results which are given in \Cref{MainTheorem4} when we choose $c=1$ and $k=1,$ coincide with the results given in \cite[Thms. 6 and 7]{ABO}.
\end{remark}


\begin{thebibliography}{width}
\bibitem[ABO18]{ABO}
I. Akta\c{s}, \'A. Baricz and H. Orhan, Bounds for radii of starlikeness and convexity of some special functions, \textit{Turk J Math}, \textbf{42}, 211--226, 2018.

\bibitem[ABY17]{ABY}
I. Akta\c{s}, \'A. Baricz and N. Ya\u{g}mur,  Bounds for the radii of
univalence of some special functions, \textit{Math. Inequal. Appl.}, \textbf{20}(3), 825--843, 2017.	
	
\bibitem[BKS14]{BKS}
\'{A}. Baricz, P.A. Kup\'{a}n and R. Sz\'{a}sz,  The radius of
starlikeness of normalized Bessel functions of the first kind, \textit{Proc. Amer. Math. Soc.}, \textbf{142}(6), 2019--2025, 2014.

\bibitem[BDOY16]{BDOY}
\'A. Baricz, D.K. Dimitrov, H. Orhan and N. Ya\u{g}mur,
Radii of starlikeness of some special functions, \textit{Proc. Amer. Math. Soc.} \textbf{144}, 3355--3367, 2016.

\bibitem[BOS16]{BOS}
\'{A}. Baricz, H. Orhan and R. Sz\'{a}sz,  The radius of $\alpha-$convexity of normalized Bessel functions of the first kind, \textit{Comput. Methods Funct. Theory}, \textbf{16}(1), 93--103, 2016.

\bibitem[BP19]{BP}
\'{A}. Baricz, A. Prajapati, Radii of starlikeness and convexity of generalized Mittag-Leffler functions, arXiv:1901.04333 

\bibitem[BSz14]{BSz14}
Baricz \'{A}. and Sz\'asz R., The radius of convexity of normalized Bessel functions of the first kind, \textit{Anal. Appl.}, \textbf{12}(5), 485--509, 2014.

\bibitem[BSz15]{BSz15}
Baricz \'{A}. and Sz\'asz R., The radius of convexity of normalized Bessel functions, \textit{Anal. Math.}, \textbf{41}(3), 141--151, 2015.

\bibitem[BSz16]{BSz16}
Baricz \'{A}. and Sz\'asz R., Close-to-convexity of some special functions, \textit{Bull. Malay. Math Sci. Soc.}, \textbf{39}(1), 427--437, 2016.	

\bibitem[BS18]{BS18}
Baricz \'{A}. and Sanjeev S., Zeros of some special entire functions, \textit{Proc. Amer. Math. Soc.}, \textbf{146}(5), 2207--2216, 2018.

\bibitem[BTK18]{BTK}
\'{A}. Baricz, E. Toklu and E.  Kad{\i}o\u{g}lu, Radii of starlikeness and convexity of Wright functions, \textit{Math. Commun.}, \textbf{23}, 97--117, 2018.	

\bibitem[Br60]{Brown}
Brown R.K., Univalence of Bessel functions. {\em Proc. Amer. Math. Soc.} \textbf{11}, 278--283, 1960.

\bibitem[DP07]{DP07}
Diaz R., Pariguan E., On hypergeometric functions and Pochhammer $k-$symbol, {\it Divulg. Math.}, \textbf{15}, 179--192, 2007.

\bibitem[DC2009]{DC}
D.K. Dimitrov and Y.B. Cheikh, Laguerre polynomials as Jensen polynomials of Laguerre-P\'olya entire functions, \textit{J. Comput. Appl. Math.}, \textbf{233}, 703--707, 2009.

\bibitem[MNR13]{MNR13}
Mubeen S., Naz M., Rahman G., A note on $k-$hypergeometric differential equations, {\it J. Inequal. Spec. Funct.}, \textbf{4}(3), 38--43, 2013.

\bibitem[MA18]{MA18}
Mondal S.R., Akel M.S, Differential equation and inequalities of the generalized $k-$Bessel functions, {\it J. Inequal. Appl.}, \textbf{175}, 1--14, 2018.

\bibitem[NP14]{NP14}
Nantomah K., Prempeh E., Some inequalities for the $k-$digamma functions, {\it Math. Æterna }, \textbf{4}, 521--525, 2014.

\end{thebibliography}
\end{document}